\documentclass[english,12pt]{article}
\usepackage[T1]{fontenc}
\usepackage[latin9]{inputenc}
\usepackage{geometry}
\usepackage{mathrsfs}
\usepackage{bm}
\usepackage{amsmath}
\usepackage{amsthm}
\usepackage{amssymb}
\usepackage{graphicx}
\usepackage{setspace}
\usepackage[authoryear]{natbib}
\makeatletter
\newcommand{\lyxaddress}[1]{
	\par {\raggedright #1
	\vspace{1.4em}
	\noindent\par}
}
\theoremstyle{plain}
\newtheorem{lem}{\protect\lemmaname}
\theoremstyle{plain}
\newtheorem{prop}{\protect\propositionname}
\theoremstyle{plain}
\newtheorem{thm}{\protect\theoremname}
\theoremstyle{remark}
\newtheorem{rem}{\protect\remarkname}

\makeatother

\usepackage{babel}
\providecommand{\lemmaname}{Lemma}
\providecommand{\propositionname}{Proposition}
\providecommand{\remarkname}{Remark}
\providecommand{\theoremname}{Theorem}

\begin{document}
\title{Concentration-based confidence intervals for U-statistics}
\author{Hien D. Nguyen}
\maketitle

\lyxaddress{Department of Mathematics and Statistics, La Trobe University, Melbourne
Victoria, Australia. Email: h.nguyen5@latrobe.edu.au.}
\begin{abstract}
Concentration inequalities have become increasingly popular in machine
learning, probability, and statistical research. Using concentration
inequalities, one can construct confidence intervals (CIs) for many
quantities of interest. Unfortunately, many of these CIs require the
knowledge of population variances, which are generally unknown, making
these CIs impractical for numerical application. However, recent results
regarding the simultaneous bounding of the probabilities of quantities
of interest and their variances have permitted the construction of
empirical CIs, where variances are replaced by their sample estimators.
Among these new results are two-sided empirical CIs for U-statistics,
which are useful for the construction of CIs for a rich class of parameters.
In this article, we derive a number of new one-sided empirical CIs
for U-statistics and their variances. We show that our one-sided CIs
can be used to construct tighter two-sided CIs for U-statistics, than
those currently reported. We also demonstrate how our CIs can be used
to construct new empirical CIs for the mean, which provide tighter
bounds than currently known CIs for the same number of observations,
under various settings.
\end{abstract}

\section{Introduction}

Let $\bm{X}_{1},\dots,\bm{X}_{n}\in\mathbb{X}\subseteq\mathbb{R}^{d}$
be an independent and identically distributed (IID) random sample
from some data generating process (DGP), characterized by some probability
distribution $F$. Let $h:\mathbb{X}^{m}\rightarrow\mathbb{R}$ be
a symmetric function, in the sense that
\[
h\left(\bm{x}_{\pi_{1}\left(1\right)},\dots,\bm{x}_{\pi_{1}\left(m\right)}\right)=h\left(\bm{x}_{\pi_{2}\left(1\right)},\dots,\bm{x}_{\pi_{2}\left(m\right)}\right)\text{,}
\]
for all $\bm{\pi}_{i}^{\top}=\left(\pi_{i}\left(1\right),\dots,\pi_{i}\left(m\right)\right)\in\Pi_{m}$
($i\in\left\{ 1,2\right\} $), where $\Pi_{m}$ is the set of all
permutations of the first $m$ consecutive natural numbers. We say
that $h$ is an order $m$ symmetric kernel. Assuming that the parameter
\[
\theta\left(F\right)=\mathbb{E}_{F}h\left(\bm{X}_{1},\dots,\bm{X}_{m}\right)=\int_{\mathbb{X}}\cdots\int_{\mathbb{X}}h\left(\bm{x}_{1},\dots,\bm{x}_{m}\right)\text{d}F\left(\bm{x}_{1}\right)\dots\text{d}F\left(\bm{x}_{m}\right)\text{,}
\]
exists, we can unbiasedly estimate $\theta\left(F\right)=\theta$
via the so-called U-statistic
\[
U_{n}=U\left(\bm{X}_{1},\dots,\bm{X}_{n}\right)={n \choose m}^{-1}\sum_{\bm{\kappa}\in\mathbb{K}_{m}}h\left(\bm{X}_{\kappa\left(1\right)},\dots,\bm{X}_{\kappa\left(m\right)}\right)\text{,}
\]
where $\bm{\kappa}^{\top}=\left(\kappa\left(1\right),\dots,\kappa\left(m\right)\right)\in\mathbb{K}_{m}$
and $\mathbb{K}_{m}$ is the set of all $n!/\left[\left(n-m\right)!m!\right]$
distinct combinations of $m$ elements from the first $n$ consecutive
natural numbers.

The U-statistics were first studied in the landmark articles of \citet{Halmos:1946aa}
and \citet{Hoeffding:1948aa}. Since their introduction, a significant
body of work has been produced on the topic. Comprehensive treatments
of the topic can be found in \citet[Ch. 5]{Serfling:1980aa}, \citet{Lee:1990aa},
\citet{Koroljuk:1994aa}, and \citet{Bose:2018aa}.

In recent years, concentration inequalities have become an important
research theme in the machine learning, probability, and statistics
research, due to their range of practical and theoretical applications.
The current state of the literature is well-reported in the volumes
of \citet{Ledoux:2001aa}, \citet{Massart:2007aa}, \citet{Dubhashi:2009aa},
\citet{Boucheron:2013aa}, and \citet{Bercu:2015aa}.

The first results regarding the concentration of $U_{n}$ about its
mean value $\theta$ were those established in \citet{Hoeffding:1963aa}.
Assume that $h\in\left[a,b\right]$ (for $a,b\in\mathbb{R}$, such
that $a<b$), and denote the variance of $h$ by $\sigma^{2}=\mathbb{V}_{F}h\left(\bm{X}_{1},\dots,\bm{X}_{m}\right)$.
Then, for any $\epsilon>0$ and $m\le n$, \citet{Hoeffding:1963aa}
proved the one-sided inequalities
\begin{equation}
\Pr\left(U_{n}-\theta\ge\epsilon\right)\le\exp\left(-\frac{2\left\lfloor n/m\right\rfloor \epsilon^{2}}{\left(b-a\right)^{2}}\right)\text{ and}\label{eq: one side hoeffding}
\end{equation}
\begin{align}
\Pr\left(U_{n}-\theta\ge\epsilon\right) & \le\exp\left(-\frac{\left\lfloor n/m\right\rfloor \epsilon^{2}}{2\sigma^{2}+\left(2c/3\right)\epsilon}\right)\text{,}\label{eq: one side bernstein}
\end{align}
where $\left\lfloor z\right\rfloor =\max\left\{ \zeta\in\mathbb{Z}:\zeta\le z\right\} $
is the floor function (cf. \citealp[Prop. 2.3]{Arcones:1993aa}) and
$c=2\max\left\{ \left|a\right|,\left|b\right|\right\} $. It is procedural
to demonstrate that the right-hand sides (RHSs) of (\ref{eq: one side hoeffding})
and (\ref{eq: one side bernstein}) also upper bound $\Pr\left(\theta-U_{n}\right)$,
thus we have the absolute inequalities
\begin{equation}
\Pr\left(\left|U_{n}-\theta\right|\ge\epsilon\right)\le2\exp\left(-\frac{2\left\lfloor n/m\right\rfloor \epsilon^{2}}{\left(b-a\right)^{2}}\right)\text{ and}\label{eq: two side hoeffding}
\end{equation}
\begin{equation}
\Pr\left(\left|U_{n}-\theta\right|\ge\epsilon\right)\le2\exp\left(-\frac{\left\lfloor n/m\right\rfloor \epsilon^{2}}{2\sigma^{2}+\left(2c/3\right)\epsilon}\right)\text{,}\label{eq: two side bernstein}
\end{equation}
since $\Pr\left(\left|U_{n}-\theta\right|\ge\epsilon\right)=\Pr\left(U_{n}-\theta\ge\epsilon\right)+\Pr\left(\theta-U_{n}\ge\epsilon\right)$.

Since the establishment of (\ref{eq: two side hoeffding}) and (\ref{eq: two side bernstein}),
there had been little progress in the derivation of fundamentally
novel bounds for U-statistics. A major contribution in this direction
was due to \citet{Arcones:1995aa}. Assume that $h\in\left[0,1\right]$,
and that $\varsigma^{2}=\mathbb{V}_{F}\mathbb{E}_{F}\left[h\left(X_{1},\dots,X_{m}\right)|X_{1}\right]<\infty$
exists. Then, for any $\epsilon>0$ and $m\le n$, \citet{Arcones:1995aa}
proved that
\begin{equation}
\Pr\left(\left|U_{n}-\theta\right|\ge\epsilon\right)\le4\exp\left(-\frac{\left\lfloor n/m\right\rfloor \epsilon^{2}}{2m\varsigma^{2}+\left(2^{m+3}m^{m-1}+\left[2/3\right]m^{-2}\right)\epsilon}\right)\text{.}\label{eq: arcones}
\end{equation}

In practice, only bounds of $h$ tend to be known, regarding data
from any arbitrary DGP. Thus, without knowledge of the variances $\sigma^{2}$
or $\varsigma^{2}$, the Bernstein-type bounds (\ref{eq: one side bernstein}),
(\ref{eq: two side bernstein}), and (\ref{eq: arcones}) cannot be
used numerically. In such situations, only the Hoeffding-type bounds
(\ref{eq: one side hoeffding}) and (\ref{eq: two side hoeffding})
tend to see practical application.

Let $X_{1},\dots,X_{n}\in\left[a,b\right]$ be independent random
variables. If $\Sigma_{n}^{2}=n^{-1}\sum_{i=1}^{n}\mathbb{V}X_{i}$
and $\bar{X}_{n}=n^{-1}\sum_{i=1}^{n}X_{i}$, then for any $\epsilon>0$,
\citet{Bennett:1962aa} proved the classic Bernstein-type inequalities
\begin{equation}
\Pr\left(\bar{X}_{n}-\mathbb{E}\bar{X}_{n}\ge\epsilon\right)\le\exp\left(-\frac{n\epsilon^{2}}{\Sigma_{n}^{2}/2+\left(2c/3\right)\epsilon}\right)\label{eq: classic bernstein one-sided}
\end{equation}
and
\begin{equation}
\Pr\left(\left|\bar{X}_{n}-\mathbb{E}\bar{X}_{n}\right|\ge\epsilon\right)\le2\exp\left(-\frac{\epsilon^{2}}{\Sigma_{n}^{2}/2+\left(2c/3\right)\epsilon}\right)\text{,}\label{eq: classic bernstein two-sided}
\end{equation}
where $c=2\max\left\{ \left|a\right|,\left|b\right|\right\} $. See
also \citet[Thm. 2.28]{Bercu:2015aa}. Similar to the U-statistics
counterpart, the general inequalities (\ref{eq: classic bernstein one-sided})
and (\ref{eq: classic bernstein two-sided}) require knowledge of
the variance $\Sigma_{n}^{2}$ to be of practical use.

A primary application of concentration inequalities is the construction
of $\left(1-\delta\right)\times100\%$ confidence intervals (CIs)
for some estimator of interest, where $\delta\in\left(0,1\right)$.
Recently, empirical CIs based on the (\ref{eq: classic bernstein one-sided})
and (\ref{eq: classic bernstein two-sided}) forms, where $\Sigma_{n}^{2}$
is replaced by an estimator, have been constructed by \citet{Audibert:2009aa}
and \citet{Maurer:2009aa}. These results allow for the construction
of CIs around the mean of data with known bounds, but unknown variances.
Such bounds were applied to perform variance penalized estimation
in \citet{Maurer:2009aa} and for the analysis of multi-armed bandit
problems in \citet{Audibert:2009aa}. Other examples of applications
include racing-based online model selection \citep{Mnih:2008aa} and
classifier boosting \citep{Shivaswamy:2010aa}.

In \citet{Peel:2010aa}, analogous results to \citet{Audibert:2009aa}
and \citet{Maurer:2009aa} were obtained, for the specific context
of U-statistics. That is, empirical two-sided CIs based on the Bernstein-type
bounds (\ref{eq: two side bernstein}) and (\ref{eq: arcones}) were
constructed, where the variances $\sigma^{2}$ or $\varsigma^{2}$
were replaced by respective empirical estimators. These CIs have useful
applications in racing-based online model selection \citep{Peel:2010aa},
change detection \citep{Sakthithasan:2013aa}, and optimal treatment
allocation \citep{Liang:2018aa}. Since the work of \citet{Peel:2010aa},
\citet{Loh:2013aa} also presented an empirical two-side CI, based
on (\ref{eq: arcones}).

In this paper, we utilize inequalities (\ref{eq: one side hoeffding})
and (\ref{eq: one side bernstein}) in order to construct empirical
one-sided CIs for U-statistics. We demonstrate that these bounds are
able to produce one-sided CIs that are analogous to the two-sided
bounds of \citet{Peel:2010aa}. Furthermore, the two-sided form of
our construction provides tighter bounds than those obtained by \citet{Peel:2010aa}.
Using our constructed CIs, we also demonstrate how one can obtain
empirical CIs for the mean from a recently-derived variance-dependent
improved Hoeffding-type concentration inequality of \citet{Bercu:2015aa}.
Our work can be seen as an extension and refinement of the results
of \citet{Peel:2010aa} and as an addition to the literature on empirical
variance-dependent bounds, as pioneered by \citet{Audibert:2009aa}
and \citet{Maurer:2009aa}. We make numerous comments regarding the
relationship between our work and previously obtained outcomes in
the final section of the article.

The paper proceeds as follows. In Section 2, we present the main results
of the paper. New empirical CIs are derived in Section 3. Concluding
remarks regarding our exposition and results are presented in Section
4.

\section{Main results}

\subsection{Technical preliminaries}

We begin the presentation of our main results by providing a pair
of lemmas that are used throughout the remainder of the paper.
\begin{lem}
[Union bound]\label{lem union bound}For random variables $X$, $Y$,
$Z$,
\[
\Pr\left(X>Z\right)\le\Pr\left(X>Y\right)+\Pr\left(Y>Z\right)\text{.}
\]
\end{lem}
\begin{proof}
For any events $\mathscr{A}$, $\mathscr{B}$, we have the union bound:
\[
\Pr\left(\mathscr{A}\cup\mathscr{B}\right)\le\Pr\left(\mathscr{A}\right)+\Pr\left(\mathscr{B}\right)\text{.}
\]
Consider that $\left(X>Z\right)$ is a subset of $\left(X>Y\right)\cup\left(Y>Z\right)$.
Thus
\[
\Pr\left(X>Z\right)\le\Pr\left(\left(X>Y\right)\cup\left(Y>Z\right)\right)\le\Pr\left(X>Y\right)+\Pr\left(Y>Z\right)\text{.}
\]
\end{proof}
\begin{lem}
[Inequality reversal]\label{lem inequality reversal}Suppose that
$X$ is a random variable, and let $A,B>0$ and $C,D\ge0$, such that
for every $\epsilon>0$,
\[
\Pr\left(X\ge\epsilon\right)\le A\exp\left(-\frac{B\epsilon^{2}}{C+D\epsilon}\right)\text{.}
\]
Then, with probability at least $1-\delta$, we have
\[
X\le\sqrt{\frac{C}{B}\log\frac{A}{\delta}}+\frac{D}{B}\log\frac{A}{\delta}\text{.}
\]
\end{lem}
\begin{proof}
Let $\Pr\left(X\ge\epsilon\right)=\delta$, then we have
\[
\delta\le A\exp\left(-\frac{B\epsilon^{2}}{C+D\epsilon}\right)\text{,}
\]
which we can solve for $\epsilon$, to get
\[
\epsilon\le\frac{1}{2B}\left(D\log\frac{A}{\delta}+\sqrt{D^{2}\log^{2}\frac{A}{\delta}+4BC\log\frac{A}{\delta}}\right)\text{.}
\]
We then apply the square root inequality $\sqrt{x+y}\le\sqrt{x}+\sqrt{y}$
in order to obtain the desired result.
\end{proof}
Although the two preceding results appear elsewhere in the literature,
we include them for the convenience of the reader. Lemma \ref{lem inequality reversal}
appears as Lemma 1 in \citet{Peel:2010aa}.

\subsection{\label{subsec:The-sample-variance}Concentration of the variance}

Let $X_{1},\dots,X_{n}\in\mathbb{R}$ be IID random variables and
note that the symmetric kernel
\[
h\left(x_{1},x_{2}\right)=\left(x_{1}-x_{2}\right)^{2}/2
\]
corresponds to the U-statistic
\[
U_{n}={n \choose 2}^{-1}\sum_{\bm{\kappa}\in\mathbb{K}_{2}}\frac{\left(X_{\kappa\left(1\right)}-X_{\kappa\left(2\right)}\right)^{2}}{2}=\left(n-1\right)^{-1}\sum_{i=1}^{n}\left(X_{i}-\bar{X}_{n}\right)^{2}=S_{n}^{2}\text{,}
\]
which is the unbiased estimator of the variance. That is $\theta=\mathbb{E}_{F}S_{n}^{2}=\mathbb{V}_{F}X$,
where $X$ arises from the same DGP as the sample $X_{1},\dots,X_{n}$.
Furthermore, assume that $X\in\left[0,1\right]$, so that $h\left(X_{1},X_{n}\right)\in\left[0,1/2\right]$.
Using (\ref{eq: one side hoeffding}), we can obtain our first result.
\begin{prop}
\label{prop Hoeffding bounds for the variance} If $X_{1},\dots,X_{n}\in\left[0,1\right]$
are IID, then for any $\epsilon>0$,
\begin{equation}
\Pr\left(S_{n}^{2}-\mathbb{V}_{F}X\ge\epsilon\right)\le\exp\left(-8\left\lfloor n/2\right\rfloor \epsilon^{2}\right)\text{.}\label{eq: var hoeffding 1}
\end{equation}
Alternatively, with probability at least $1-\delta$,
\begin{equation}
S_{n}^{2}-\mathbb{V}_{F}X\le\sqrt{\frac{1}{8\left\lfloor n/2\right\rfloor }\log\frac{1}{\delta}}\text{.}\label{eq: var hoeffding ci 1}
\end{equation}
Both (\ref{eq: var hoeffding 1}) and (\ref{eq: var hoeffding ci 1})
hold with $S_{n}^{2}-\mathbb{V}_{F}X$ replaced by $\mathbb{V}_{F}X-S_{n}^{2}$.
\end{prop}
\begin{proof}
Result (\ref{eq: var hoeffding 1}) is obtained by direct substitution
$S_{n}^{2}$ into (\ref{eq: one side hoeffding}), and (\ref{eq: var hoeffding ci 1})
arises via Lemma \ref{lem inequality reversal}.
\end{proof}
Assume the same hypothesis as Proposition \ref{prop Hoeffding bounds for the variance}.
We consider instead a bound for $S_{n}^{2}$ using (\ref{eq: one side bernstein}).
Since $h\left(X_{1},X_{2}\right)\le1/2$, it is true that
\begin{equation}
\mathbb{V}_{F}S_{n}^{2}=\mathbb{E}_{F}\left[S_{n}^{4}\right]-\mathbb{E}_{F}\left[S_{n}^{2}\right]\le\mathbb{E}_{F}S_{n}^{2}=\mathbb{V}_{F}X\text{,}\label{eq: variance less than mean}
\end{equation}
and thus, by direct substitution into (\ref{eq: one side bernstein}),
we obtain the bound
\[
\Pr\left(S_{n}^{2}-\mathbb{V}_{F}X\ge\epsilon\right)\le\exp\left(-\frac{\left\lfloor n/2\right\rfloor \epsilon^{2}}{2\mathbb{V}_{F}X+\left(2/3\right)\epsilon}\right)\text{,}
\]
since $c=\max\left\{ 0,1\right\} =1$.

Using Lemma \ref{lem inequality reversal}, we have, with probability
at least $1-\delta$,

\[
S_{n}^{2}-\mathbb{V}_{F}X\le\sqrt{\frac{2\mathbb{V}_{F}X}{\left\lfloor n/2\right\rfloor }\log\frac{1}{\delta}}+\frac{2}{3\left\lfloor n/2\right\rfloor }\log\frac{1}{\delta}\text{,}
\]
which we can complete the square to obtain
\begin{equation}
S_{n}^{2}\le\left[\sqrt{\mathbb{V}_{F}X}+\frac{1}{2}\sqrt{\frac{2}{\left\lfloor n/2\right\rfloor }\log\frac{1}{\delta}}\right]^{2}+\frac{1}{6\left\lfloor n/2\right\rfloor }\log\frac{1}{\delta}\text{.}\label{eq: intermediate 1}
\end{equation}
Taking the square root, and applying the square root inequality to
both sides of (\ref{eq: intermediate 1}) yields

\begin{align*}
\sqrt{S_{n}^{2}} & +\sqrt{\frac{1}{6\left\lfloor n/2\right\rfloor }\log\frac{1}{\delta}}\le\sqrt{\mathbb{V}_{F}X}+\sqrt{\frac{2}{4\left\lfloor n/2\right\rfloor }\log\frac{1}{\delta}}\\
 & \le\sqrt{\mathbb{V}_{F}X}+\left(\frac{\sqrt{2}}{2}+\frac{\sqrt{6}}{6}\right)\sqrt{\frac{1}{\left\lfloor n/2\right\rfloor }\log\frac{1}{\delta}}\text{.}
\end{align*}

From (\ref{eq: one side bernstein}), we also have, with probability
at least $1-\delta$,
\[
\mathbb{V}_{F}X-S_{n}^{2}\le\sqrt{\frac{2\mathbb{V}_{F}X}{\left\lfloor n/2\right\rfloor }\log\frac{1}{\delta}}+\frac{1}{3\left\lfloor n/2\right\rfloor }\log\frac{1}{\delta}\text{,}
\]
which rearranges to

\[
\mathbb{V}_{F}X-\sqrt{\frac{2\mathbb{V}_{F}X}{\left\lfloor n/2\right\rfloor }\log\frac{1}{\delta}}-\frac{1}{3\left\lfloor n/2\right\rfloor }\log\frac{1}{\delta}\le S_{n}^{2}
\]
and we can again complete the square to obtain

\[
\left[\sqrt{\mathbb{V}_{F}X}-\frac{1}{2}\sqrt{\frac{2}{\left\lfloor n/2\right\rfloor }\log\frac{1}{\delta}}\right]^{2}\le S_{n}^{2}+\frac{7}{6\left\lfloor n/2\right\rfloor }\log\frac{1}{\delta}\text{.}
\]
By the square root inequality, we have
\begin{align*}
\sqrt{\mathbb{V}_{F}X} & \le\sqrt{S_{n}^{2}}+\left(\frac{\sqrt{2}}{2}+\frac{\sqrt{42}}{6}\right)\sqrt{\frac{1}{\left\lfloor n/2\right\rfloor }\log\frac{1}{\delta}}\text{.}
\end{align*}
We therefore have the following empirical one-sided CIs.
\begin{prop}
\label{prop variance bernstein}If $X_{1},\dots,X_{n}\in\left[0,1\right]$
are IID, then with probability at least $1-\delta$,
\begin{equation}
\sqrt{S_{n}^{2}}\le\sqrt{\mathbb{V}_{F}X}+\left(\frac{\sqrt{2}}{2}+\frac{\sqrt{6}}{6}\right)\sqrt{\frac{1}{\left\lfloor n/2\right\rfloor }\log\frac{1}{\delta}}\text{ and }\label{eq: Prop 2 1}
\end{equation}
\begin{equation}
\sqrt{\mathbb{V}_{F}X}\le\sqrt{S_{n}^{2}}+\left(\frac{\sqrt{2}}{2}+\frac{\sqrt{42}}{6}\right)\sqrt{\frac{1}{\left\lfloor n/2\right\rfloor }\log\frac{1}{\delta}}\text{,}\label{eq: Prop 2 2}
\end{equation}
where $\sqrt{2}/2+\sqrt{6}/6\le1.116$ and $\sqrt{2}/2+\sqrt{42}/6\le1.788$.
\end{prop}

\subsection{Variance of a U-statistic}

Assume that $h\in\left[0,1\right]$. Following \citet{Peel:2010aa},
we introduce a new kernel of order $2m$:
\[
\eta\left(\bm{x}_{1},\dots,\bm{x}_{2m}\right)=\left[h\left(\bm{x}_{1},\dots,\bm{x}_{m}\right)-h\left(\bm{x}_{m+1},\dots,\bm{x}_{2m}\right)\right]^{2}/2\text{.}
\]

We note that $\eta$ can either be symmetric or otherwise. If $\eta$
is not symmetric, then we can define a symmetric version of $\eta$,
in the form
\[
\tilde{\eta}\left(\bm{x}_{1},\dots,\bm{x}_{2m}\right)=\frac{1}{\left(2m\right)!}\sum_{\bm{\pi}\in\Pi_{2m}}\eta\left(\bm{x}_{\pi\left(1\right)},\dots,\bm{x}_{\pi\left(2m\right)}\right)\text{,}
\]
where $\mathbb{E}_{F}\left[\eta\left(\bm{X}_{1},\dots,\bm{X}_{2m}\right)\right]=\mathbb{E}_{F}\left[\tilde{\eta}\left(\bm{X}_{1},\dots,\bm{X}_{2m}\right)\right]$
(cf. \citealp[Ch. 5]{Serfling:1980aa}). Furthermore, we can inspect
that
\begin{align*}
\mathbb{E}_{F}\eta & =\frac{1}{2}\mathbb{E}_{F}\left[h\left(\bm{X}_{1},\dots,\bm{X}_{m}\right)-h\left(\bm{X}_{m+1},\dots,\bm{X}_{2m}\right)\right]^{2}\\
 & =\frac{1}{2}\left[\mathbb{E}_{F}h^{2}-2\mathbb{E}_{F}h\mathbb{E}_{F}h+\mathbb{E}_{F}h^{2}\right]\\
 & =\mathbb{E}_{F}h^{2}-\left[\mathbb{E}_{F}h\right]^{2}=\mathbb{V}_{F}h=\sigma^{2}\text{,}
\end{align*}
and note also that $h\in\left[0,1\right]$ implies $\eta\in\left[0,1/2\right]$.

Consider the U-statistics
\[
W_{n}={n \choose 2m}^{-1}\sum_{\bm{\kappa}\in\mathbb{K}_{2m}}\eta\left(\bm{X}_{\kappa\left(1\right)},\dots,\bm{X}_{\kappa\left(2m\right)}\right)
\]
and
\[
\tilde{W}_{n}={n \choose 2m}^{-1}\sum_{\bm{\kappa}\in\mathbb{K}_{2m}}\tilde{\eta}\left(\bm{X}_{\kappa\left(1\right)},\dots,\bm{X}_{\kappa\left(2m\right)}\right)\text{.}
\]
Using $W_{n}$ and $\tilde{W}_{n}$, we seek to obtain $\left(1-\delta\right)\times100\%$
one-sided CIs for the comparison of the quantities $\mathbb{V}_{F}h$
and $W_{n}$.

Assume that $\eta$ is a symmetric kernel. As in Section \ref{subsec:The-sample-variance},
we may use (\ref{eq: one side hoeffding}) to obtain the following
result, analogous to Proposition \ref{prop Hoeffding bounds for the variance}.
\begin{prop}
\label{prop Hoeffding bounds for the U-stat var} If $\bm{X}_{1},\dots,\bm{X}_{n}\in\mathbb{X}$
are IID, $h\in\left[0,1\right]$, and $\eta$ is symmetric, then for
any $\epsilon>0$,
\begin{equation}
\Pr\left(W_{n}-\sigma^{2}\ge\epsilon\right)\le\exp\left(-8\left\lfloor n/\left(2m\right)\right\rfloor \epsilon^{2}\right)\text{.}\label{eq: var hoeffding U-stat}
\end{equation}
Alternatively, with probability at least $1-\delta$,
\begin{equation}
W_{n}-\sigma^{2}\le\sqrt{\frac{1}{8\left\lfloor n/\left(2m\right)\right\rfloor }\log\frac{1}{\delta}}\text{.}\label{eq: var hoeffding ci U-stat}
\end{equation}
Both (\ref{eq: var hoeffding U-stat}) and (\ref{eq: var hoeffding ci U-stat})
hold with $W_{n}-\sigma^{2}$ replaced by $\sigma^{2}-W_{n}$.
\end{prop}
Unfortunately, we cannot always assume that $\eta$ is symmetric.
However, by definition, $\tilde{\eta}$ is always symmetric and has
the same range as $\eta$. That is, if $h\in\left[0,1\right]$, then
$\tilde{\eta}\in\left[0,1/2\right]$. Thus, we have the following
Proposition.
\begin{prop}
\label{cor Hoeffding bounds for the U-stat asym} If $\bm{X}_{1},\dots,\bm{X}_{n}\in\mathbb{X}$
are IID and $h\in\left[0,1\right]$, then all of the conclusions from
Proposition \ref{prop Hoeffding bounds for the U-stat var} hold with
$W_{n}$ replaced by $\tilde{W}_{n}$.
\end{prop}
In order to obtain Bernstein bound analogs of Proposition \ref{prop Hoeffding bounds for the U-stat var}
and Proposition (\ref{cor Hoeffding bounds for the U-stat asym}),
we require the following versions of inequality \ref{eq: variance less than mean}:
\[
\mathbb{V}_{F}\eta=\mathbb{E}_{F}\left[\eta^{2}\right]-\left[\mathbb{E}_{F}\eta\right]^{2}\le\mathbb{E}_{F}\left[\eta^{2}\right]\le\mathbb{E}_{F}\eta=\sigma^{2}\text{ and }
\]
\[
\mathbb{V}_{F}\tilde{\eta}=\mathbb{E}_{F}\left[\tilde{\eta}^{2}\right]-\left[\mathbb{E}_{F}\tilde{\eta}\right]^{2}\le\mathbb{E}_{F}\left[\tilde{\eta}^{2}\right]\le\mathbb{E}_{F}\tilde{\eta}=\sigma^{2}\text{.}
\]
In the same manner in which Proposition \ref{prop variance bernstein}
was obtained, we may prove the following result.
\begin{prop}
\label{prop symmetric bernstein variance}If $\bm{X}_{1},\dots,\bm{X}_{n}\in\mathbb{X}$
are IID, $h\in\left[0,1\right]$, and $\eta$ is symmetric, then with
probability at least $1-\delta$,
\begin{equation}
\sqrt{W_{n}}\le\sqrt{\sigma^{2}}+\left(\frac{\sqrt{2}}{2}+\frac{\sqrt{6}}{6}\right)\sqrt{\frac{1}{\left\lfloor n/\left(2m\right)\right\rfloor }\log\frac{1}{\delta}}\text{ and }\label{eq: sym bern1}
\end{equation}
\begin{equation}
\sqrt{\sigma^{2}}\le\sqrt{W_{n}}+\left(\frac{\sqrt{2}}{2}+\frac{\sqrt{42}}{6}\right)\sqrt{\frac{1}{\left\lfloor n/\left(2m\right)\right\rfloor }\log\frac{1}{\delta}}\text{,}\label{eq: sym bern2}
\end{equation}
where $\sqrt{2}/2+\sqrt{6}/6\le1.116$ and $\sqrt{2}/2+\sqrt{42}/6\le1.788$.
More generally, (\ref{eq: sym bern1}) and (\ref{eq: sym bern2})
also hold with $W_{n}$ replaced by $\tilde{W}_{n}$.
\end{prop}

\subsection{Empirical confidence intervals}

Assume that $\bm{X}_{1},\dots,\bm{X}_{n}\in\mathbb{X}$ are IID and
$h\in\left[0,1\right]$. Via inequality (\ref{eq: one side bernstein})
and Lemma \ref{lem inequality reversal}, we have
\begin{equation}
U_{n}-\theta>\sqrt{\frac{2\sigma^{2}}{\left\lfloor n/m\right\rfloor }\log\frac{2}{\delta}}+\frac{4}{3\left\lfloor n/m\right\rfloor }\log\frac{2}{\delta}\text{,}\label{eq: delta over 2 union}
\end{equation}
with probability at most $\delta/2$. If $\eta$ is symmetric, then
Proposition \ref{prop Hoeffding bounds for the U-stat var} and Proposition
\ref{cor Hoeffding bounds for the U-stat asym} imply that
\begin{equation}
\sigma^{2}>W_{n}+\sqrt{\frac{1}{8\left\lfloor n/\left(2m\right)\right\rfloor }\log\frac{2}{\delta}}\text{,}\label{eq: hoeffding u stat delta on 2}
\end{equation}
with probability at most $\delta/2$. We may apply Lemma \ref{lem union bound}
along with the square root inequality in order to prove that
\begin{equation}
U_{n}-\theta\le\sqrt{\frac{2W_{n}}{\left\lfloor n/m\right\rfloor }\log\frac{2}{\delta}}+\sqrt{\frac{1}{\left\lfloor n/m\right\rfloor }\sqrt{\frac{1}{2\left\lfloor n/\left(2m\right)\right\rfloor }}\log^{3/2}\left(\frac{2}{\delta}\right)}+\frac{4}{3\left\lfloor n/m\right\rfloor }\log\frac{2}{\delta}\text{,}\label{eq: Hoeffding empirical bound}
\end{equation}
with probability at least $1-\delta$. Inequality (\ref{eq: Hoeffding empirical bound})
also holds with $U_{n}-\theta$ replaced by $\theta-U_{n}$. Furthermore,
(\ref{eq: Hoeffding empirical bound}) also holds with $W_{n}$ replaced
by $\tilde{W}_{n}$.

From Proposition \ref{prop Hoeffding bounds for the U-stat var},
we have
\[
\sqrt{\sigma^{2}}>\sqrt{W_{n}}+\left(\frac{\sqrt{2}}{2}+\frac{\sqrt{42}}{6}\right)\sqrt{\frac{1}{\left\lfloor n/\left(2m\right)\right\rfloor }\log\frac{2}{\delta}}\text{,}
\]
with probability at most $\delta/2$, for symmetric $\eta$. Using
Lemma \ref{lem union bound} in combination with (\ref{eq: delta over 2 union}),
we obtain the bound
\begin{equation}
U_{n}-\theta\le\sqrt{\frac{2W_{n}}{\left\lfloor n/m\right\rfloor }\log\frac{2}{\delta}}+\left(\frac{\sqrt{2}}{2}+\frac{\sqrt{42}}{6}\right)\sqrt{\frac{2}{\left\lfloor n/m\right\rfloor \left\lfloor n/\left(2m\right)\right\rfloor }}\log\frac{2}{\delta}+\frac{4}{3\left\lfloor n/m\right\rfloor }\log\frac{2}{\delta}\text{,}\label{eq: Bernstein empirical bound}
\end{equation}
with probability at least $1-\delta$. Generally, (\ref{eq: Bernstein empirical bound})
also holds when we replace $U_{n}-\theta$ by $\theta-U_{n}$ or $W_{n}$
by $\tilde{W}_{n}$, or both simultaneously. We summarize the results
of this section in the following theorem.
\begin{thm}
If $\bm{X}_{1},\dots,\bm{X}_{n}\in\mathbb{X}$ are IID, $h\in\left[0,1\right]$,
and $\eta$ is symmetric, then inequalities (\ref{eq: Hoeffding empirical bound})
and (\ref{eq: Bernstein empirical bound}) hold with probability at
least $1-\delta$. More generally, both (\ref{eq: Hoeffding empirical bound})
and (\ref{eq: Bernstein empirical bound}) hold with probability at
least $1-\delta$ when $U_{n}-\theta$ is replaced by $\theta-U_{n}$,
when $W_{n}$ is replaced by $\tilde{W}_{n}$, or when both quantities
are substituted, simultaneously.
\end{thm}

\section{Empirical confidence intervals based on an improved Hoeffding inequality}

For independent random variables $X_{1},\dots,X_{n}\in\left[a,b\right]$,
the inequalities
\begin{equation}
\Pr\left(\bar{X}_{n}-\mathbb{E}\bar{X}_{n}\ge\epsilon\right)\le\exp\left(-\frac{2n\epsilon^{2}}{\left(b-a\right)^{2}}\right)\text{ and }\label{eq: classic hoeffding 1side}
\end{equation}
\begin{equation}
\Pr\left(\left|\bar{X}_{n}-\mathbb{E}\bar{X}_{n}\right|\ge\epsilon\right)\le2\exp\left(-\frac{2n\epsilon^{2}}{\left(b-a\right)^{2}}\right)\label{eq: classic hoeffding 2side}
\end{equation}
were proved, for any $\epsilon>0$, in \citet{Hoeffding:1963aa}.
In \citet{Bercu:2015aa}, an interesting improvement to the Hoeffding
inequality of form (\ref{eq: classic hoeffding 1side}) was reported.
We present the IID expectation form of the inequality below, and note
that the more general summation form appears as Theorem 2.47 in \citet{Bercu:2015aa}.
\begin{thm}
[Bercu et al., 2015] \label{thm bercu}If $X_{1},\dots,X_{n}\in\left[a,b\right]$
are IID random variables, then
\begin{equation}
\Pr\left(\bar{X}_{n}-\mathbb{E}_{F}X\ge\epsilon\right)\le\exp\left(-\frac{3n\epsilon^{2}}{\left(b-a\right)^{2}+2\mathbb{V}_{F}X}\right)\text{.}\label{eq: bercu 2015}
\end{equation}
Furthermore, (\ref{eq: bercu 2015}) also holds when $\bar{X}_{n}-\mathbb{E}_{F}X$
is replaced by $\mathbb{E}_{F}X-\bar{X}_{n}$.
\end{thm}
Without loss of generality, suppose that $\left[a,b\right]=\left[0,1\right]$.
Then, we obtain
\[
\Pr\left(\bar{X}_{n}-\mathbb{E}_{F}X\ge\epsilon\right)\le\exp\left(-\frac{3n\epsilon^{2}}{1+2\mathbb{V}_{F}X}\right)
\]
and, with probability at least $1-\delta$,

\[
\bar{X}_{n}-\mathbb{E}_{F}X\le\sqrt{\frac{1+2\mathbb{V}_{F}X}{3n}\log\frac{1}{\delta}}\text{,}
\]
via Lemma \ref{lem inequality reversal}. Thus, with probability at
most $\delta/2$,
\begin{equation}
\bar{X}_{n}-\mathbb{E}_{F}X>\sqrt{\frac{1+2\mathbb{V}_{F}X}{3n}\log\frac{2}{\delta}}\text{.}\label{eq: improved hoeffding delta on 2}
\end{equation}
Similar to (\ref{eq: hoeffding u stat delta on 2}), we have
\begin{equation}
\mathbb{V}_{F}X>S_{n}^{2}+\sqrt{\frac{1}{8\left\lfloor n/2\right\rfloor }\log\frac{2}{\delta}}\label{eq: hoeffding var delta on 2}
\end{equation}
with probability at most $\delta/2$, via (\ref{eq: var hoeffding ci 1}).

Combining (\ref{eq: improved hoeffding delta on 2}) and (\ref{eq: hoeffding var delta on 2})
via Lemma \ref{lem union bound} and the square root inequality then
yields
\begin{equation}
\bar{X}_{n}-\mathbb{E}_{F}X\le\sqrt{\frac{1+2S_{n}^{2}}{3n}\log\frac{2}{\delta}}+\sqrt{\frac{1}{12n}\sqrt{\frac{8}{\left\lfloor n/2\right\rfloor }}\log^{3/2}\left(\frac{2}{\delta}\right)}\text{,}\label{eq: empirical improved hoeffding-hoeffding}
\end{equation}
with probability at least $1-\delta$. Given the symmetry of (\ref{eq: bercu 2015}),
we may also switch $\bar{X}_{n}-\mathbb{E}_{F}X$ with $\mathbb{E}_{F}X-\bar{X}_{n}$
in (\ref{eq: empirical improved hoeffding-hoeffding}).

Next, (\ref{eq: Prop 2 2}) implies that
\begin{equation}
\sqrt{\mathbb{V}_{F}X}>\sqrt{S_{n}^{2}}+\left(\frac{\sqrt{2}}{2}+\frac{\sqrt{42}}{6}\right)\sqrt{\frac{1}{\left\lfloor n/2\right\rfloor }\log\frac{2}{\delta}}\label{eq: bernstein var delta on 2}
\end{equation}
with probability at most $\delta/2$. Combining with (\ref{eq: bernstein var delta on 2})
via Lemma \ref{lem union bound} and the square root inequality then
yields

\begin{equation}
\bar{X}_{n}-\mathbb{E}_{F}X\le\sqrt{\frac{1}{3n}\log\frac{2}{\delta}}+\sqrt{\frac{2S_{n}^{2}}{3n}\log\frac{2}{\delta}}+\left(\frac{\sqrt{2}}{2}+\frac{\sqrt{42}}{6}\right)\sqrt{\frac{2}{3\left\lfloor n/2\right\rfloor n}}\log\frac{2}{\delta}\text{,}\label{eq: empirical improved hoeffding-bernstein}
\end{equation}
with probability at least $1-\delta$. Again, the inequality (\ref{eq: empirical improved hoeffding-bernstein})
holds if we switch $\bar{X}_{n}-\mathbb{E}_{F}X$ with $\mathbb{E}_{F}X-\bar{X}_{n}$.

Finally, note that we can obtain two-sided versions of (\ref{eq: empirical improved hoeffding-hoeffding})
and (\ref{eq: empirical improved hoeffding-bernstein}) by considering
the union of lower bounds on $\bar{X}_{n}-\mathbb{E}_{F}X$, $\mathbb{E}_{F}X-\bar{X}_{n}$,
and $\sqrt{\mathbb{V}_{F}X}$, simultaneously. We then obtain, with
probability at least $1-\delta$,
\[
\left|\bar{X}_{n}-\mathbb{E}_{F}X\right|\le\sqrt{\frac{1+2S_{n}^{2}}{3n}\log\frac{4}{\delta}}+\sqrt{\frac{1}{12n}\sqrt{\frac{8}{\left\lfloor n/2\right\rfloor }}\log^{3/2}\left(\frac{3}{\delta}\right)}
\]
and
\[
\left|\bar{X}_{n}-\mathbb{E}_{F}X\right|\le\sqrt{\frac{1}{3n}\log\frac{3}{\delta}}+\sqrt{\frac{2S_{n}^{2}}{3n}\log\frac{3}{\delta}}+\left(\frac{\sqrt{2}}{2}+\frac{\sqrt{42}}{6}\right)\sqrt{\frac{2}{3\left\lfloor n/2\right\rfloor n}}\log\frac{3}{\delta}\text{.}
\]

\section{Concluding remarks}
\begin{rem}
Inequalities (\ref{eq: one side bernstein}), (\ref{eq: two side bernstein}),
(\ref{eq: classic bernstein one-sided}), and (\ref{eq: classic bernstein two-sided})
are often stated with either the additional assumption that $\mathbb{E}_{F}h=0$
or that $\mathbb{E}\bar{X}_{n}=0$ (see, e.g., \citealp[Prop. 2.3]{Arcones:1993aa},
\citet{Arcones:1995aa}, and \citealp[Thm. 2.28]{Bercu:2015aa}).
As such, the constant $c$ is usually defined as $\max\left\{ \left|a\right|,\left|b\right|\right\} $.
However we define $c=2\max\left\{ \left|a\right|,\left|b\right|\right\} $,
since we allow for cases where $\mathbb{E}_{F}h\ne0$ or $\mathbb{E}\bar{X}_{n}\ne0$,
and due to the fact that $\left|X-\mathbb{E}X\right|\le\left|X\right|+\left|\mathbb{E}X\right|\le2\max\left\{ \left|a\right|,\left|b\right|\right\} $,
for any $X\in\left[a,b\right]$.
\end{rem}
\begin{rem}
To the best of our knowledge, Propositions \ref{prop Hoeffding bounds for the variance}
and \ref{prop Hoeffding bounds for the U-stat var} and Proposition
\ref{cor Hoeffding bounds for the U-stat asym} do not appear elsewhere
in the literature. Proposition \ref{prop variance bernstein} is a
special case of Proposition \ref{prop symmetric bernstein variance},
where we use the order one kernel $h=x$. Proposition \ref{prop variance bernstein}
can be viewed as a refinement of intermediate results from the proof
of \citet[Thm. 3]{Peel:2010aa}. Our refinements are as follows: we
correctly differentiated the cases where $\eta$is symmetric or where
we must replace it by $\tilde{\eta}$; we constructed our CIs using
the inequality (\ref{eq: one side bernstein}) instead of the two-sided
version (\ref{eq: two side bernstein}), and thus obtained better
constants; we obtained one-sided CIs for both $\sqrt{W_{n}}-\sqrt{\sigma^{2}}$
and $\sqrt{\sigma^{2}}-\sqrt{W_{n}}$ (and when $W_{n}$ is replaced
by $\tilde{W}_{n}$, whereas \citet{Peel:2010aa} only considered
the CI for $\sqrt{\sigma^{2}}-\sqrt{W_{n}}$; and lastly, we make
no assumptions on the divisibility of $n$, whereas \citet{Peel:2010aa}
assumes that $n$ is divisible by both 2 and $m$ in their presentation.
\end{rem}
\begin{rem}
We may compare Proposition \ref{prop variance bernstein} directly
to Theorem 10 of \citet{Maurer:2009aa}, which cannot be obtained
within the U-statistics framework. Under the same conditions as Proposition
\ref{prop variance bernstein}, \citet[Thm. 10]{Maurer:2009aa} states
that, with probability at least $1-\delta$,
\begin{equation}
\sqrt{S_{n}^{2}}\le\sqrt{\mathbb{V}_{F}X}+\sqrt{\frac{2}{n-1}\log\frac{1}{\delta}}\text{ and}\label{eq: Maurer 1}
\end{equation}
\begin{equation}
\sqrt{\mathbb{V}_{F}X}\le\sqrt{S_{n}^{2}}+\sqrt{\frac{2}{n-1}\log\frac{1}{\delta}}\text{.}\label{eq: Maurer 2}
\end{equation}
It is easy to see that the CIs of \citet[Thm. 10]{Maurer:2009aa}
achieve the same rates with respect to $n$ and $\delta$ as Proposition
\ref{prop variance bernstein}. Furthermore, for large $n$, (\ref{eq: Maurer 1})
and (\ref{eq: Maurer 2}) provide tighter bounds than (\ref{eq: Prop 2 1})
and (\ref{eq: Prop 2 2}), respectively, for almost all values of
$n$. We note that the only exception is when $n=2$ or $4$, where
(\ref{eq: Prop 2 1}) is tighter than (\ref{eq: Maurer 1}).
\end{rem}
\begin{rem}
By considering the lower bounds of $U_{n}-\theta$, $\theta-U_{n}$
and $\sqrt{\sigma^{2}}$, simultaneously, we may obtain the two-sided
version of (\ref{eq: Bernstein empirical bound}):
\begin{equation}
\left|U_{n}-\theta\right|\le\sqrt{\frac{2W_{n}}{\left\lfloor n/m\right\rfloor }\log\frac{3}{\delta}}+\left(\frac{\sqrt{2}}{2}+\frac{\sqrt{42}}{6}\right)\sqrt{\frac{2}{\left\lfloor n/m\right\rfloor \left\lfloor n/\left(2m\right)\right\rfloor }}\log\frac{3}{\delta}+\frac{4}{3\left\lfloor n/m\right\rfloor }\log\frac{3}{\delta}\text{,}\label{eq: twosided bernstein Ustat}
\end{equation}
with probability at least $1-\delta$. We may compare this directly
with \citet[Thm. 3]{Peel:2010aa}, which for symmetric $\eta$ and
under the assumption that $n$ is divisible by $2$ and $m$, implies
that
\[
\left|U_{n}-\theta\right|\le\sqrt{\frac{2mW_{n}}{n}\log\frac{4}{\delta}}+\frac{5m}{n}\log\frac{4}{\delta}\text{,}
\]
with probability at least $1-\delta$. Under the same assumptions
regarding the divisibility of $n$, we can write (\ref{eq: twosided bernstein Ustat})
as
\[
\left|U_{n}-\theta\right|\le\sqrt{\frac{2mW_{n}}{n}\log\frac{3}{\delta}}+\frac{\left(4+\sqrt{2}\left[3+\sqrt{21}\right]\right)m}{3n}\log\frac{3}{\delta}\text{,}
\]
where $\left(4+\sqrt{2}\left[3+\sqrt{21}\right]\right)/3\le4.908$.
Thus, (\ref{eq: twosided bernstein Ustat}) is tighter than the CI
of \citet[Thm. 3]{Peel:2010aa}.
\end{rem}
\begin{rem}
As noted in \citet[Sec. 2.5.4]{Bercu:2015aa}, (\ref{eq: bercu 2015})
is a strict improvement of (\ref{eq: classic hoeffding 1side}) since
$\left(b-a\right)^{2}\ge4\mathbb{V}_{F}X$. Furthermore, it is provable
that $\left(b-a\right)^{2}>4\mathbb{V}_{F}X$ for all random variables
$X\in\left[a,b\right]$, except for binary $X\in\left\{ a,b\right\} $,
where $\Pr\left(X=a\right)=\Pr\left(X=b\right)=1/2$.
\end{rem}
\begin{rem}
CIs (\ref{eq: empirical improved hoeffding-hoeffding}) and (\ref{eq: empirical improved hoeffding-bernstein})
may be compared directly to the empirical Bernstein-type inequalities
of \citet{Audibert:2009aa} and and \citet{Maurer:2009aa}. Under
the same conditions as those under which (\ref{eq: empirical improved hoeffding-hoeffding})
and (\ref{eq: empirical improved hoeffding-bernstein}) are established,
with probability at least $1-\delta$, we have
\begin{equation}
\bar{X}_{n}-\mathbb{E}_{F}X\le\sqrt{\frac{2\left(n-1\right)S_{n}^{2}}{n^{2}}\log\left(\frac{2}{\delta}\right)}+\frac{3}{n}\log\left(\frac{2}{\delta}\right)\text{ and}\label{eq: audibert bennett}
\end{equation}
\begin{equation}
\bar{X}_{n}-\mathbb{E}_{F}X\le\sqrt{\frac{2S_{n}^{2}}{n}\log\left(\frac{2}{\delta}\right)}+\frac{7}{3\left(n-1\right)}\log\left(\frac{2}{\delta}\right)\text{,}\label{eq: maurer bennett}
\end{equation}
via \citet[Thm. 1]{Audibert:2009aa} and \citet[Thm. 4]{Maurer:2009aa},
respectively. A visual comparison of the logarithms of the RHSs of
(\ref{eq: empirical improved hoeffding-hoeffding}), (\ref{eq: empirical improved hoeffding-bernstein}),
(\ref{eq: audibert bennett}), and (\ref{eq: maurer bennett}) is
provided in Figure \ref{fig: AudibertMaurer}. We compare the four
bounds for $S_{n}^{2}\in\left\{ 0.05,0.25\right\} $ and $\delta\in\left\{ 0.01,0.1\right\} $.
It is observable that (\ref{eq: empirical improved hoeffding-hoeffding})
was uniformly tighter than (\ref{eq: empirical improved hoeffding-bernstein}).
For smaller values of $S_{n}^{2}$, (\ref{eq: audibert bennett})
and (\ref{eq: maurer bennett}) were tighter than (\ref{eq: empirical improved hoeffding-hoeffding}),
for larger $n$. For larger $S_{n}^{2}$, (\ref{eq: empirical improved hoeffding-hoeffding})
was tighter than (\ref{eq: audibert bennett}) and (\ref{eq: maurer bennett})
over a middle range of $n$, however (\ref{eq: empirical improved hoeffding-bernstein})
remained uncompetitive. Changing $\delta$ did not tend to alter the
relative performance of the bounds. As noted by \citet{Maurer:2009aa},
(\ref{eq: maurer bennett}) has better constants than (\ref{eq: audibert bennett})
and thus provided tighter bounds for larger values of $n$. We finally
note that (\ref{eq: empirical improved hoeffding-bernstein}) can
be improved by using (\ref{eq: Maurer 2}) in the place of (\ref{eq: Prop 2 2})
in its derivation. This was not pursued because we wished for the
derivation to be self-contained within the U-statistics framework.

\begin{figure}
\begin{centering}
\includegraphics[width=16cm]{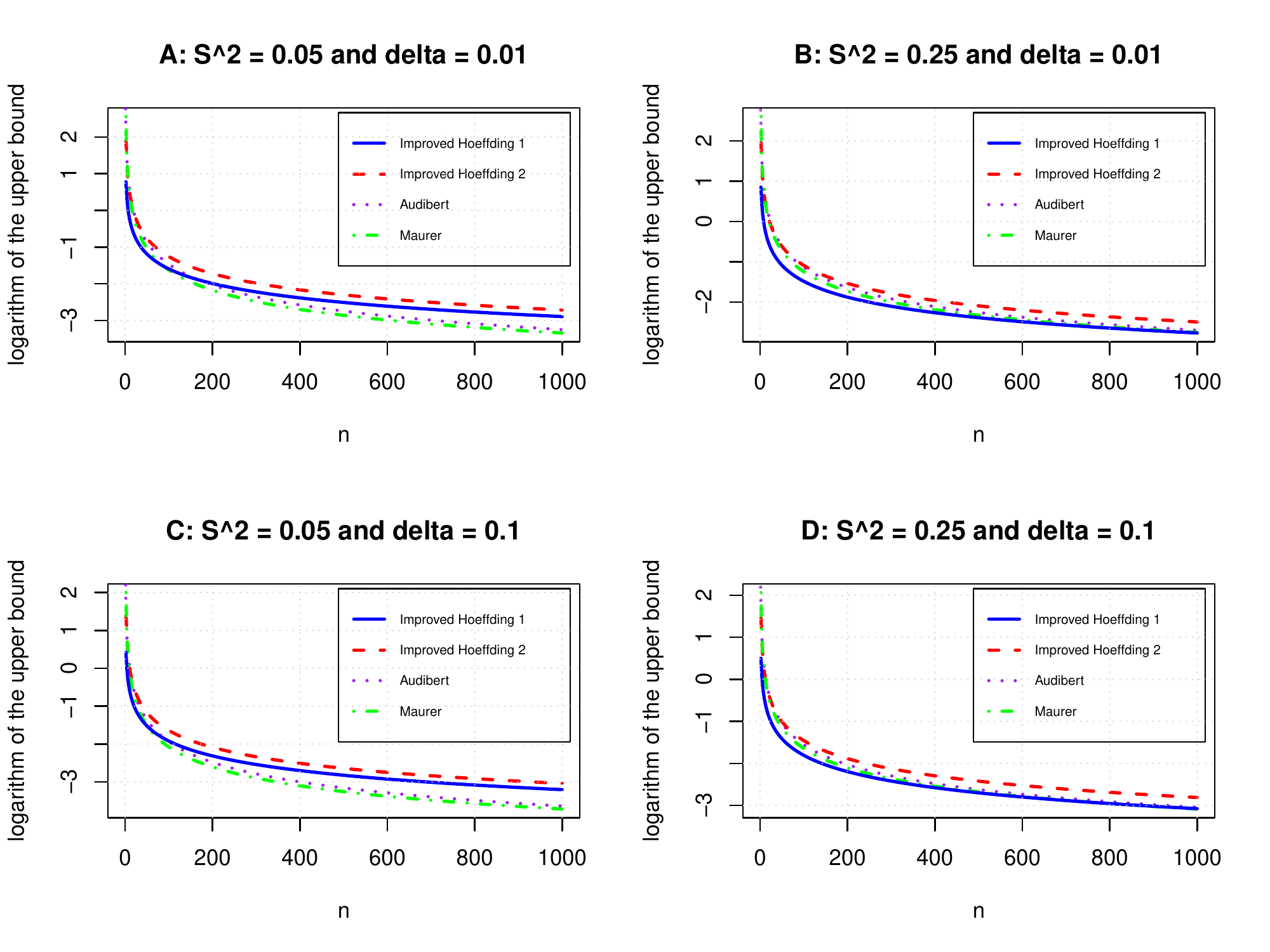}
\par\end{centering}
\caption{\label{fig: AudibertMaurer} Logarithms of the upper bounds of the
CIs (\ref{eq: empirical improved hoeffding-hoeffding}), (\ref{eq: empirical improved hoeffding-bernstein}),
(\ref{eq: audibert bennett}), and (\ref{eq: maurer bennett}), as
a function of $n$, for $S_{n}^{2}$ and $\delta$ set at various
levels. Each subplot A--D visualizes the relative performances of
the four bounds, for the values $S_{n}^{2}$ and $\delta$ that are
displayed in the title. The labels Improved Hoeffding 1 and 2 correspond
to (\ref{eq: empirical improved hoeffding-hoeffding}) and (\ref{eq: empirical improved hoeffding-bernstein}),
respectively, whereas Audibert and Maurer correspond to (\ref{eq: audibert bennett})
and (\ref{eq: maurer bennett}), respectively.}

\end{figure}
\end{rem}
\begin{rem}
For $k,n\in\mathbb{N}$, such that $k<n$, we have $\left\lfloor n/k\right\rfloor >\left(n-k+1\right)/k$.
Thus, we may remove the floor operator in each of the inequalities
where it appears by upper bounding its multiplicative inverse. For
instance,
\[
\sqrt{\frac{2mW_{n}}{n-m+1}\log\frac{3}{\delta}}+\left(\frac{\sqrt{2}}{2}+\frac{\sqrt{42}}{6}\right)\sqrt{\frac{4m^{2}}{\left(n-m+1\right)\left(n-2m+1\right)}}\log\frac{3}{\delta}+\frac{4m}{3\left(n-m+1\right)}\log\frac{3}{\delta}
\]
is an upper bound for the RHS of (\ref{eq: twosided bernstein Ustat}).
\end{rem}

\section*{Acknowledgements}

The author is funded by Australian Research Council grants DE170101134
and DP180101192.

\bibliographystyle{apalike2}
\bibliography{20190210_Ustat}

\end{document}